\documentclass[10pt]{article}

\usepackage{caption}
 \usepackage{geometry}
\usepackage{amsmath,amsfonts,amssymb,amsthm,mathrsfs}
\usepackage{graphicx}
\usepackage{stmaryrd}
\usepackage{hyperref}
\usepackage{tikz}

\theoremstyle{definition}
\newtheorem{definition}{Definition}
\newtheorem{remark}[definition]{Remark}

\theoremstyle{mytheorem}
\newtheorem{theorem}[definition]{Theorem}

\newtheorem{proposition}[definition]{Proposition}

\newcommand{\equlaw}{\stackrel{(d)}{=}}
\newcommand{\var}{\text{\rm{\bf Var}}}

\newcommand{\Cov}{\text{\rm{\bf Cov}} }
\newcommand{\Var}{\var}

\renewcommand{\P}{\mathbf  P}
\newcommand{\E}{\mathbf  E}

\newcommand{\sinc}{\text{\rm{sinc}}}
\author{Rapha\"el Lachi\`eze-Rey\footnote{Universit\'e de Paris, Laboratoire MAP5, UMR CNRS 8145, 45 Rue des Saints-P\`eres, F-75006 Paris, France.}}
\title{Variance linearity for real Gaussian zeros}
\begin{document}
\maketitle 

{\bf Abstract}
We investigate the zero set of a stationary Gaussian process on the real line, and in particular give  lower bounds for the variance of the number of points and of linear statistics on a large interval, in all generality. We prove that this point process is  never hyperuniform, i.e. the variance is at least linear, and give a necessary condition to have linear variance, which is close to be sufficient. We study the class of symmetric Bernoulli convolutions and give an example where the zero set is   maximally rigid,  weakly mixing, and not hyperuniform.\\

%
%
%
 
 
 
 \section{Background and motivation}

We study here  the zero set of a  Gaussian stationary process $\mathsf  X$:
\begin{align*}
\mathsf  Z_{\mathsf  X}:=\{t\in \mathbb{R}: \mathsf  X(t)=0\}.
\end{align*}
The study of Gaussian zeros has emerged in the fifties, with pioneering works of Kac and Rice \cite{Kac43}, and subsequent applications in the fields of telecommunications and signal processing, ocean waves, and random mechanics.  Rigorous results can be found in the book of Cramer and Leadbetter  \cite{CL67}, along with the first second order results.  In the following decades, most significant second order analyses have been made by   Cuzick \cite{Cuz}, Slud \cite{Slud91}, Kratz \& L\'eon \cite{KL01}, see the survey \cite{Kratz-survey} and references therein. The weakest available condition for non-degenerate asymptotic variance is the square integrability of the covariance function, in \cite{Slud91}.
 We give here a variance lower bound, under virtually no hypotheses, that implies  that the variance is always at least linear, and give a necessary condition very close to known sufficient conditions for it to be actually linear. Linear statistics on Gaussian nodal sets have also been the topic of many recent results, see for instance \cite{Wigman10} on the sphere, or \cite{BKPV,NS11} for zeros of Gaussian entire functions; these statistics can be used in particular to prove the rigidity of a point process, see  \cite{GP17}. We also explore the class of symmetric Bernoulli convolutions and focus on an example whose spectral measure has a Cantor nature. 
 
 Feldheim \cite{Feldheim18} has  studied zeros of Gaussian analytic functions (GAFs) whose law is invariant only under horizontal shifts, reduced to a horizontal strip of the form $\mathbb{R}\times [a,b]$, with $-\infty <a<b<\infty $.  In a related work, Buckley and Feldheim \cite{BF18} study the winding number of a GAF $\mathsf X:\mathbb{R}\mapsto \mathbb  C$. 
A strong motivation of their work is the analogy between the winding number and the number of zeros.  
Some convenient identities coming from complex analysis can provide explicit expressions for these indexes and their moments, whereas the number of zeros of a real Gaussian process is mostly studied in the literature through its Wiener-Ito expansion, which is sometimes not amenable to analysis; one can also use directly  Kac-Rice formula \cite{AW} but this does not ease the task of estimating the variance. In particular, we cannot obtain an expression as explicit as (6) in \cite{BF18} or  Section 3.2 in \cite{Feldheim18} for the number of zeros.
It will be interesting to observe in Section \ref{sec:Gaussian-zeros} that the results  about  GAFs are similar to the results we obtain here: the variance of the number of zeros (or the winding number) is always at least linear, it is not linear if some square integrability conditions related to the covariance functions are not satisfied, and it is quadratic if the spectral measure has atoms. Very recently, Aza\"is and Dalmao \cite{AD20} have studied the winding number of planar Gaussian fields on the real line, realising the Wiener-Ito counterpart to the work \cite{BF18}, obtaining similar results without analyticity assumptions.

Another motivation is to explore the behaviour of Gaussian zeros in the light of the concepts of hyperuniformity and rigidity. Let us introduce more formally point processes before going further. Let
 $
\mathscr  N(\mathbb{R}) 
$
be the space of locally finite subsets of $\mathbb{R}$, endowed with the $\sigma $-algebra $\mathscr{B}(\mathscr  N(\mathbb{R}))$ generated by the mappings
\begin{align*}
\varphi _{A}:\mathsf  Z\in \mathscr  N(\mathbb{R})\mapsto \# \mathsf Z\cap A
\end{align*}
for $A$ a Borel subset of $\mathbb{R}$, where $\#$ denotes the cardinality of a set.  A {\it point process} $\mathsf Z$ is a measurable mapping from an underlying probability space $(\Omega ;\mathscr {A},\P )$ to $(\mathscr  N(\mathbb{R}),\mathscr  B(\mathscr  N(\mathbb{R})))$; $\mathsf Z$ is furthermore stationary if $\mathsf Z+x\equlaw \mathsf Z$ for $x\in \mathbb{R},$ and we denote by $V_\mathsf Z (T)=\Var(\# \mathsf Z\cap [-T,T])$. 
In physics of condensed matters and statistical physics, a great attention has been given recently to the phenomenon of {\it  fluctuation suppression} of particle systems, which can be incarnated by  the property of {\it hyperuniformity}:
a stationary point process $\mathsf Z$ on $\mathbb{R}$ is {\it hyperuniform}, or {\it superhomegeneous},  if 
\begin{align*}
\liminf_{T\to \infty }\frac{V_\mathsf Z (T)}{T}=0.
\end{align*}
A striking property of many systems with suppressed fluctuations  is their  {\it rigidity}, i.e. the complete determination of a feature of the point process in a bounded domain given its configuration outside the domain:
for a measurable mapping $\varphi :\mathscr  N(\mathbb{R})\to \mathbb{R}$, we say a point process $\mathsf Z$ is $\varphi ${\it -rigid} if for all $T>0$, 
\begin{align*}
\varphi (\mathsf Z\cap [-T,T])\in \sigma ( \mathsf Z\cap [-T,T] ^{c}).
\end{align*}
In particular, we say it is {\it number-rigid} if it is $\varphi $-rigid for $\varphi =\#$, and {\it maximally rigid} if it is $\varphi $-rigid for every bounded measurable $\varphi $, i.e. if $\mathsf Z\cap [-T, T]$ is completely determined by $\mathsf Z\cap [-T,T]^{c}$ for $T>0$.

This phenomenon has been much documented, with many rigourous results proven in dimension $1$ and $2$, and recently  in dimension $3$  by Chatterjee on hierarchichal Coulomb gases \cite{Cha17}. The link between  rigidity and hyperuniformity is the topic of many works in the recent years.
 Most notable examples of stationary hyperuniform models are determinantal processes   whose kernels are projection operators \cite{GL17}, zeros of planar Gaussian analytic functions \cite{GP17}, some Coulomb systems \cite{GL17}, the $\beta $-sine processes for $\beta >0$ \cite{DHLM}, perturbed lattices \cite{PS14}, and also many non-stationary models, with a more general definition of hyperuniformity. In all these models, the point process is hyperuniform and $\varphi $-rigid, at least for $\varphi (\mathsf Z)=\#\mathsf Z$.   It is widely believed that hyperuniformity is required to have rigidity \cite{GL17}, and the usual  method of proof of rigidity, introduced in \cite{GP17}, involves linear statistics and implictly assumes hyperuniformity  (see \cite{DHLM} for an alternative method involving DLR equations). The phenomenon of rigidity, although mathematically intriguing, has also practical uses  in some percolation models related to the underlying point process, see \cite{GKP}.

We exploit the variance results derived in this paper to exhibit a non artificial example of a point process in dimension $1$ that experiences the strongest possible form of rigidity, without being hyperuniform. Such examples are actually not so hard to build on $\mathbb{R}^{d}$, take for instance the grid
\begin{align*}
\mathsf Z_{d}=L(\mathbb{ Z}^{d} +U)
\end{align*}
where $L$ is some non-trivial non-negative $L^{2}$ real random variable and $U$ is independent and uniformly distributed on $[0,1)^{d}$.  It is also clear that there is absolutely no form of asymptotic independence in $\mathsf Z_{d}$, in the sense that given two events $A,B$, the events $\mathsf Z_{d}\in  A$ and $\mathsf Z+t\in  B$ have no reason to be uncorrelated for large $T$. For this reason, we might ask for additional desirable properties of such counter-examples. 

\begin{definition}We say that a stationary point process $\mathsf Z$ of $\mathbb{R}$ is  {\it weakly mixing} if for all events $A,B\in  \mathscr{B} (\mathscr  N(\mathbb{R})),$
\begin{align*}
\lim_{T\to \infty }T^{-1}\int_{0}^{T}|\P (\mathsf Z\in A,(\mathsf Z+t)\in B)-\P (\mathsf Z\in A)\P (\mathsf Z\in B)|dt= 0.
\end{align*}
\end{definition}
 The randomly scaled grid $\mathsf  Z_{1}$ does not satisfy this property because it is not even ergodic.
We give here a stationary Gaussian process $\mathsf  X$ whose spectral measure belongs to the class of {\it symmetric Bernoulli convolutions} and whose zero set $\mathsf  Z_{\mathsf  X}$    is weakly mixing, maximally rigid, and not hyperuniform.  See Theorem \ref{thm:main-variance},  Propositions \ref{prop:factorial} and  \ref{prop:variance-factorial}.

Regarding rigidity and hyperuniformity, stationary point processes obtained as the zeros of a random Gaussian functions have mainly been studied for {\it Gaussian analytic functions}   in the complex plane, see \cite{BKPV}. The joint requirements of complex Gaussianity, analycity, and stationarity have reduced the class of zeros to essentially one model, up to rescaling, which happens to be rigid and hyperuniform  \cite{BKPV}. 
The links between Gaussian zeros and rigidity is yet essentially to discover.\\

We present in Section \ref{sec:Gaussian} some results and definitions about Gaussian processes and their zeros, variance lower bounds, and we give a condition for maximal rigidity. We also properly explore the analogies with variance results for GAFs. In Section \ref{sec:bernoulli} we introduce the class of symmetric Bernoulli convolutions, give some properties and  examples, in particular an example yielding maximally rigid non hyperuniform zeros. In Section \ref{sec:chaotic} we prove the lower bounds on the variance by using chaotic decompositions of the number of zeros.

\section{Gaussian zeroes}
\label{sec:Gaussian}
We introduce here basic notions about Gaussian processes. See the monograph of Aza\"is \& Wschebor \cite{AW} for background material and advanced results about Gaussian processes and their level sets.
A real Gaussian process is a random function $\mathsf X:\mathbb{R}\to \mathbb{R}$ for which any finite dimensional distribution $(\mathsf X(t_{1}),\dots ,\mathsf X(t_{n})),t_{1},\dots ,t_{n}\in \mathbb{R}$, is a Gaussian vector with covariance matrix $(\mathcal  C(t_{i},t_{j}))_{1\leqslant i,j\leqslant n}$, where $\mathcal  C$ is the {\it  covariance function} of the process $\mathsf X$.  The law of $\mathsf X$ is entirely determined by $\mathcal  C$, and the stationarity of $\mathsf X$, defined by $\mathsf X(\cdot +t)\equlaw \mathsf X$ for $t\in \mathbb{R}$, is characterised by the existence of a function $C_{\mathsf X}$, called {\it reduced covariance function}, such that $\mathcal  C(t,s)=C_{\mathsf X}(t-s),t,s\in \mathbb{R}$. Such a function is known to admit a spectral representation of the form 
\begin{align*}
C_{\mathsf X}(t)=\int_{\mathbb{R}}e^{ixt}\mu_{\mathsf X} (dx)=2\int_{0}^{\infty }\cos(xt)\mu _{\mathsf X}(dx),t\in \mathbb{R},
\end{align*}
where $\mu _{\mathsf X}$ is the (symmetric) {\it spectral measure} of $\mathsf X$ on $\mathbb{R}$.
We say that $\mathsf X$ is {\it degenerate} if $C_{\mathsf X}(x)=A\cos(\varphi x)$ for some $A\geqslant 0,\varphi \in \mathbb{R}$, or equivalently if $\mu _{\mathsf X}$ is formed by two symmetric atoms with same mass.

\subsection{Zero set}

\label{sec:Gaussian-zeros}

We are primarily interested in the variance $V_{\mathsf  Z _{\mathsf X}}(T)$ of $\#(\mathsf  Z _{\mathsf X}\cap [-T,T])$, denoted by $V_{\mathsf X}(T)$ for short, where $X$ is a Gaussian stationary process.
 The Kac-Rice formula, established in the fourties \cite{Kac43}, is the starting point of most  rigourous subsequent works. Cramer and Leadbetter \cite{CL67}  proved  that it is necessary for $V_{\mathsf X}(T)$  to be finite for every $T>0$ that $C_{\mathsf X}$ is twice differentiable in $0$  and for some $\delta >0$, 
\begin{align}
\label{eq:geman}
\int_{0}^{\delta }\frac{1}{t^{2}}(C_{\mathsf X}'(t)-C_{\mathsf X}''(0)t)dt<\infty .
\end{align}Geman \cite{Geman} then proved that the condition is actually necessary on any non-negligible interval, and this condition is thus referred to as {\it Geman's condition.}   
When this condition is satisfied, we use without loss of generality the convention $C_{\mathsf X}''(0)=-1$, equivalent to replacing $\mathsf X(t)$ by $\mathsf X(\alpha t)$ with $\alpha =(-C_{\mathsf X}''(0))^{-1/2}$.

By using approximation by $m$-dependent fields, Cuzick \cite{Cuz}  proves that the variance is not super-linear, that is $\lim_{T\to \infty }T^{-1}V_{\mathsf X}(T)$ exists and is finite, if Geman's condition holds, as well as
\begin{align}
\label{eq:C2-L2}
\int_{}C_{\mathsf X}(t)^{2}dx<\infty ,\int_{}C_{\mathsf X}''(t)^{2}dx<\infty.
\end{align}
 He actually proves a central limit theorem under the additional assumption that $$\lim_{T\to \infty }T^{-1}V_{\mathsf X}(T)>0.$$ Computing the  decomposition of the number of zeros with respect to Wiener-Ito integrals based on the field $\mathsf X$, Slud \cite{Slud91}  proves that \eqref{eq:C2-L2} is sufficient for the asymptotic variance to be finite and positive, and for the central limit theorem to hold. Kratz \& L\'eon \cite{KL01} compute the chaos decomposition with respect to the joint process formed by the field $\mathsf X$ and its derivative $\mathsf X'$ and ease some computations related to the decomposition, allowing them to generalise the results to other crossing problems; this is the approach we will use at Section \ref{sec:chaotic}.

We present here our results regarding the variance lower bound for the number of zeros. The method extends  to linear statistics 
\begin{align*}
N_\mathsf  X({\varphi }):=\sum_{z\in \mathsf  X}\varphi (z)
\end{align*} for $\varphi $ in the space $\mathcal{C}_{b} $ of compactly supported piecewise continuous functions $\varphi $ such that $\int_{\mathbb{R}}\varphi (x)dx\neq 0$. We are interested in the asymptotics of $N_\mathsf  X({\varphi _{T}})$ for $\varphi _{T}(\cdot )=\varphi (T^{-1}\cdot )$ as $T\to \infty $. We abuse the notation $V_{\mathsf  X}(\varphi )=\Var(N_\mathsf  X({\varphi }))$, with $V_{\mathsf  X}(T)=V_{\mathsf  X}(\varphi _{T})$ if $\varphi $ is the indicator function of $[-1,1]$. Say that $\mu _{X}$ is {\it regular} in $x\in \mathbb{R}$ if it admits a $L^{2}$ density on some neighbourhood of $x.$ 

\begin{theorem}
\label{thm:main-variance}
Let $\mathsf X$ be a non-degenerate stationary Gaussian field on $\mathbb{R}$, $\varphi \in \mathcal{C}_{b}$ . Assume either that $\varphi $ is non-negative, or that Geman's condition is satisfied. Then
\begin{itemize}
\item[(i)] The variance is at least linear, i.e. $$\liminf_{T> {0}}T^{-1}V_{\mathsf X}(\varphi _{T})\in (0,\infty ].$$
\item[(ii)] the condition $C''+C\in L^{2}(\mathbb{R})$ is  necessary for linear fluctuations, i.e. for $$\liminf_{T\to \infty }T^{-1}V_{\mathsf X}(\varphi _{T})<\infty. $$
If $\mu_{X}$ is regular in $\pm 1$, the condition $C''+C\in L^{2}(\mathbb{R})$ is equivalent to \eqref{eq:C2-L2}, which is sufficient for linear fluctuations of $V_{\mathsf  X}(T)$ (\cite{Cuz}).
\item[ (iii)]If $\mu _{\mathsf X}$ has an atom in $\mathbb{R}\setminus \{-1,1\}$, $\inf_{T> 0}T^{-2}V_{\mathsf X}(\varphi _{T})>0.$
\end{itemize}
\end{theorem}The proof is  at Section \ref{sec:chaotic}.
We also give   an example of a singular nonatomic spectral measure $\mu _{\mathsf X}$ for which  $T^{2-\varepsilon }=o(V_{\mathsf X}(T))$ for $\varepsilon >0.$ Regarding the quadratic variance, as soon as the variance is finite on some interval with non-empty interior (Geman's condition \eqref{eq:geman}), then the variance cannot be more than quadratic, see \cite{AW}.

{\bf Related works.}
Even in the case $V_{\mathsf  X}(\varphi _{T})=V_{\mathsf  X}(T)$, it is not clear which of these results already appeared   in full generality, and in particular when $\mu $ does not have a density. We could not locate a  necessary condition other than \eqref{eq:C2-L2} for linear variance in the literature, see also Remark 1.2 in \cite{Feldheim18}.  It remains open to find a necessary and sufficient condition. In many practical studies of Gaussian zeros, the covariance is assumed to be square integrable (\cite{AL20,AW,Cuz,KL01,Slud91}), it is in particular square integrable in the neighbourhood of $\pm 1$, hence the linearity is equivalent to $C,C''\in L^{2}$, and to $C+C''\in L^{2}.$

For positive asymptotic variance to hold ((i) with $\varphi =\mathbf{1}_{ [-1,1] } $), we can notice that Cuzick \cite{Cuz} needs
\begin{align*}
\int_{}\frac{(C'(t))^{2}}{1-C(t)^{2}}dt<\frac{\pi }{2}\sqrt{\int_{}x^{2}\mu (dx)}
\end{align*}
for $\lim_{T\to \infty }T^{-1}V_{\mathsf X}(T)$ to be strictly positive, a requirement still present in \cite{Kratz-survey}. Other works aimed at proving central limit theorems with linear variance give various types of conditions:  $C\in L^{2}$ in \cite{Slud91}, $C\in L^{1},C^{''}\in L^{2} $  in \cite{KL01}.

Linear statistics are classic functionals of interest for point processes, see for instance   \cite{NS11} which focuses on complex Gaussian zeros, and references therein, as well as the variance computed in \cite{GP17} aimed at proving rigidity results. For linear statistics of real Gaussian zeros, we were notified while revising the current work that Ancona and Letendre \cite{AL20} obtained 
asymptotic results under some integrability hypotheses on the reduced covariance function. The condition $\int_{\mathbb{R}}\varphi (x)dx\neq 0$ does not appear in their results, but it is necessary in the current general framework, and in particular in the periodic case (point (iii)).\\
 
The assumption that $\mathsf  X$ is not degenerate cannot be removed, for if $C_{\mathsf  X}(t)=\cos(t)$, the variance is bounded (see the proof of Theorem \ref{thm:main-variance}). The presence of a singularity at $\pm 1$ muddies the waters anyway, for instance if $C_{\mathsf  X}(t)=\frac{\cos(t)+  \varepsilon e^{-t^{2}/2}}{1+\varepsilon }$ for some $\varepsilon >0$, the variance is at least linear.\\

{\bf Comparison with zeros of GAFs.}
It is interesting to compare the results above to those obtained by Feldheim \cite{Feldheim18}, who studied the zero set $\mathsf  Z _{\mathsf X}^{[a,b]}$ of a Gaussian analytic function (GAF) $\mathsf X$ on the infinite strip 
\begin{align*}
\mathbb  S_{a,b}=\{x+\imath y:x\in \mathbb{R},a\leqslant y\leqslant b\},
\end{align*} 
for some $-\infty <a<b<\infty $, and those of Buckley and Feldheim \cite{BF18} who study the winding number  $W_{T}$ of a GAF $\mathsf X:[0,T]\mapsto \mathbb  C$.   A GAF $\mathsf X$ is also characterised by its reduced covariance $C_{\mathsf X}$ and its spectral measure $\rho_{\mathsf X} $ along the real line:
\begin{align*}
C_{\mathsf X}(t-s)=\E (\mathsf X(t)\overline{\mathsf X}(s))=\int_{\mathbb{R}}e^{-2\pi i(t-s)x}\rho _{\mathsf X} (dx),t,s\in \mathbb{R}.
\end{align*}
Let us reproduce some of their results regarding the variances $$V_{\mathsf X}^{[a,b]}(T)=\Var(\# \mathsf  Z _{\mathsf X}^{[a,b]}\cap ([-T,T]\times [a,b])),\hspace{1cm}V'_{\mathsf X}(T)=\Var (W_{T}).$$ The field $\mathsf X$ is called {\it degenerate} if $\rho _{\mathsf X } $ consists of a single atom.

\begin{theorem}[\cite{Feldheim18,BF18}]
\label{thm:var-GAF}
Let $\mathsf X$ be a non-degenerate GAF.
\begin{itemize}\item [(i)] If the differentiability condition (3) in \cite{BF18} holds, for $T>0$
\begin{align*}
T^{-1}V'_{\mathsf X}(T)\geqslant C \text{\rm{ for some }}C>0.\\
\end{align*}
The left hand member has a limit in $(0,\infty )$ as $T\to \infty $ if $C_{\mathsf X},C_{\mathsf X}'\in L^{2}$.
\item[(ii)]The limit
\begin{align*}
 \lim_{T\to \infty }T^{-1}V_{\mathsf X}^{[a,b]}(T)
\end{align*}
 exists and is strictly positive. It is finite if  
\begin{align*}
\int_{} | C_{\mathsf X}(t+iy) | ^{2}dt<\infty ,\int_{} | C_{\mathsf X}''(t+iy) | ^{2}dt<\infty  ,y\in \{a,b\}
\end{align*}
where $C_{\mathsf X}$ is analytically continued to $\mathbb  S_{a,b}$,
\item [(iii)] the limits
\begin{align*}
\lim_{T\to \infty }T^{-2}V_{\mathsf X}^{[a,b]}(T),\;\;\;\lim_{T\to \infty }T^{-2}V'_{\mathsf X} (T) 
\end{align*}exist  and are finite, and are strictly positive iff $\rho _{\mathsf X} $ has an atom.

\end{itemize}
\end{theorem}

Hence the results are similar to ours regarding the quadratic variance and the linear variance. Feldheim \cite{Feldheim18} also gives a lower bound which yields a sufficient condition for asymptotic super linear variance, but leaves linearity undecided for some measures with a singularity, such as for $\rho_{\mathsf X} (dx)=\mathbf{1}_{\{ | x | \leqslant 1\}} | x | ^{-1/2}dt,x\in \mathbb{R}$. This is analoguous to our ``dark spot'' regarding linear variance, at the difference that we allow singularities as long as they are not located at $\pm 1$ (under the convention $C(0)=-C''(0)=1.$)

\begin{remark}A degenerate GAF is, up to rescaling and multiplication,  $\mathsf X(t)=\gamma e^{2i\pi t}$, where $\gamma =\alpha +\imath \beta $ is a Gaussian standard complex variable (i.e. $\alpha ,\beta $ are i.i.d real Gaussian variables), whereas a degenerate real Gaussian process in the sense of this paper is of the form $\mathsf X_{0}(t)=\alpha \cos(t)+\beta \sin(t)$, up to rescaling and multiplication, and its reduced covariance function is $C_{0}(t)=\cos(t)$.

It is on the other hand difficult to infer results for zeros of Gaussian processes on $\mathbb{R}$ from results about zeros of GAF because a GAF will typically take non-real values on the real line, and hence have no real zeros (this is why the winding number might be seen as a generalisation of the number of zeros). For instance, the unique (up to rescaling/multiplication) GAF $\mathsf  X$ such that for $t \in \mathbb{R}$, $\mathbb{E}(\mathsf  X(0) \overline{{\mathsf  X(t)}})=2\cos( t)$ is $\mathsf  X=\mathsf  X_{0}+i \mathsf  X_{0}'$ where $\mathsf  X_{0}'$ is an independent copy of $\mathsf  X_{0}.$
\end{remark}

\subsection{Completely predictable real zeros}

We will actually consider a property stronger than maximal rigidity.

\begin{definition}Say that a point process $\mathsf  Z $ is {\it completely predictable} if 
\begin{align*}
\sigma (\mathsf  Z )=\cap _{t\in \mathbb{R}}\sigma (\mathsf  Z \cap (-\infty ,t)).
\end{align*}
\end{definition}

This property yields that the knowledge of $\mathsf  Z $ on any half line determines the whole process. We have the obvious fact that if $\mathsf  Z $ is completely predictable, it is maximally rigid.

The problem of wether the data of a stationary  Gaussian process on a half line $\{\mathsf X(t);t\leqslant 0\}$ determines the whole process has been widely investigated, see Szeg\"o's alternative (\cite[Section 4.2]{DMcK}). Such a behaviour does not  imply that the knowledge of the zeros on a half line determines all the zeros. One might also try to reconstruct the whole process from the zeros, but Weierstrass factorisation theorem yields uncountably many functions with a given infinite set of zeros (without accumulation point). It would still be interesting to try to characterise processes whose zero set is completely predictable, or exhibit a weaker form of rigidity.
We give here a simple condition on $C_{\mathsf X}$ for the predictability  of $\mathsf  Z _{\mathsf X}.$
\begin{proposition}Assume $\mathsf X$ is almost surely of class $\mathcal{C}^{1}$, and there is a sequence $t_{n}\to \infty $ such that $ | C_{\mathsf X}(t_{n}) | \to 1$. Then  $\mathsf  Z _{\mathsf X}$ is predicable.
\end{proposition}

\begin{proof}
Let us denote the shifted process by $\mathsf X_{n}=\mathsf X(t_{n}+\cdot )$. We have for $t\in \mathbb{R}$, $\pm \Cov (\mathsf X (t),\mathsf X_{n}(t))\to  1$, hence $\pm \mathsf X_{n}(t)\to \mathsf X(t)$ since these are Gaussian variables. 

Let now $I$ be a bounded interval of $\mathbb{R}.$ Denote by $\dot I$ the interior of $I$ and by $\bar I$ its closure.
Since $\mathsf  Z _{\mathsf X },\mathsf  Z _{\mathsf X_{n}}$ are stationary point processes, with probability $1$, they don't have zeros on $\partial I=\bar I\setminus \dot I$.  Hence a.s., for simultaneously all $n$, $$(\mathsf  Z _{\mathsf X_{n}}\cap \bar I\neq \emptyset )=(\mathsf  Z _{\mathsf X_{n}}\cap \dot I\neq \emptyset );\;(\mathsf  Z _{\mathsf X }\cap \bar I\neq \emptyset )=(\mathsf  Z _{\mathsf X }\cap \dot I\neq \emptyset ).$$

Since $(\mathsf X(0),\mathsf X'(0))$ is not degenerate,   the zeros of $\mathsf X$ are almost all regular, i.e. $\mathsf X$ and $\mathsf X'$ don't vanish simultaneously (see \cite[Th. 1.1]{Kratz-survey}). Hence for each zero $z$ of $\mathsf X$ we have $\mathsf X(t)\mathsf X(s)<0$ for some $s,t$ arbitrarily close from $z.$

Then we have  almost surely
\begin{align*}
\mathsf  Z _{\mathsf X}\cap \dot I\neq \emptyset \Rightarrow&  {\mathsf X}(t) {\mathsf X}(s)<0\text{\rm{ for some }}t,s\in \dot I\cap \mathbb  Q\\
 \Rightarrow &\exists t,s\in \dot I\cap \mathbb  Q: {\mathsf X_{n}}(s){\mathsf X_{n}}(t)<0 \text{\rm{ for $n$ sufficiently large}}
\end{align*}
because we have a.s. the simultaneous convergences $\mathsf X_{n}(t)\to \mathsf X(t),t\in \mathbb  Q$, hence we can conlude that if $ {\mathsf X}$ vanishes in $\dot I$, so does $ {\mathsf X_{n}}$ for $n$ sufficiently large.

Let us prove the converse statement. We have by Fatou's lemma, and using the stationarity
again,
\begin{align*}
\E (\liminf_{n}\|\mathsf X_{n}'\|_{I})\leqslant \liminf_{n}\E (\|\mathsf X_{n}'\|_{I})=\E (\|\mathsf X'\|_{I})
\end{align*}
and the last quantity is finite  since $\mathsf X'$ is a.s. bounded on $I$ (see \cite{AW}).
Hence
\begin{align*}
\P (\liminf_{n} | \mathsf X_{n}' | =\infty )=0
\end{align*}
and there is a.s. $L<\infty $ and $N\subset \mathbb{N}$ infinite such that $ \|\mathsf X_{n}'\|_{I}\leqslant L $ for $n\in N.$

So let us assume that $ {\mathsf X_{n}}(t_{n})=0$ with $t_{n}\in \dot I$ for $n$ sufficiently large. 
 Then it also vanishes for $n\in N$ sufficiently large, hence there is a subsequence $n'\in N$ such that $\mathsf X_{n'}(t_{n'})=0$.  Let us take another subsequence $n''$ such that $t_{n''}\to t\in \bar I$. We have 
\begin{align*}
 | \mathsf X(t) | \leqslant  | \mathsf X(t)-\mathsf X_{n''}(t) | + | \mathsf X_{n''}(t)-\mathsf X_{n''}(t_{n''}) | \leqslant  | \mathsf X(t)-\mathsf X_{n''}(t) | +L | t-t_{n''} | ,
\end{align*}
hence the latter a.s. converges to $0$, and $\mathsf X$ vanishes in $t\in \bar I$. since $\mathsf X$ does not vanish on $\partial I$ a.s., we have proved the a.s. equivalence, 
\begin{align*}
\mathsf  Z _{\mathsf X}\cap I\neq \emptyset \Leftrightarrow \mathsf  Z _{\mathsf X_{n}}\cap I\neq \emptyset \text{\rm{ for $n$ sufficiently large.}}
\end{align*}
In particular, $\mathsf  Z _{\mathsf X}\cap I\neq \emptyset $ is completely determined by $$\bigcap _{n}\sigma (\mathsf  Z _{\mathsf X_{n}}\cap I)=\bigcap _{n}\sigma ((\mathsf  Z _{\mathsf X}+t_{n})\cap I)=\bigcap _{n}\sigma (\mathsf  Z _{\mathsf X}\cap (I-t_{n})) \subset  \bigcap _{t} \sigma( (\mathsf  Z _{\mathsf X}\cap [t,\infty ))),$$
 which concludes the proof.
\end{proof}

Let us give simple examples of Gaussian processes that satisfy this assumption, see \cite{CG89} for the details. Say that a continuous function $f:\mathbb{R}\to \mathbb{R}$ is (Bohr)-{\it almost periodic} if for each $\varepsilon >0$ there exists $T_{0}(\varepsilon )>0$ such that every interval of length $T_{0}(\varepsilon )$ contains a number $\tau $ with the following property:
$ | f(t+\tau )-f(t) | <\varepsilon $ for each $t\in \mathbb{R}$. Bochner came up with the following characterisation: $f$ is almost periodic iff for every sequence $t_{n}\to \infty $, there is a subsequence $t_{n'}$ such that 
\begin{align*}
\|f-f(t_{n'}+\cdot )\|\to 0.
\end{align*} Such functions are equivalently characterised as uniform limits of trigonometric polynomials on the real line.

\begin{proposition}Let $a=(a_{i})$ be a summable sequence of non-negative numbers, $\varphi =(\varphi _{i})$ a sequence of real numbers, and 
\begin{align}
\label{eq:a-p-cov}
C(x)=\sum_{i}a_{i}\cos(\varphi _{i}x).
\end{align}
Then $C$ is almost periodic, and it is the reduced covariance $C=C_{\mathsf X}$ of a Gaussian stationary process $\mathsf X$ which is a.s. bounded, almost periodic, and whose zero set $\mathsf  Z _{\mathsf X}$ is completely predictable.
\end{proposition}

\begin{proof}$C$ is obviously the uniform limit of trigonometric polynomials. Then $\mathsf X$ can be represented by 
\begin{align*}
\mathsf X(x)=\sum_{i}a_{i}(\alpha _{i}\cos(\varphi _{i}x)+\beta _{i}\sin(\varphi _{i}x))
\end{align*}where the $\alpha _{i},\beta _{i}$ form a sequence of independent standard Gaussian variables, hence $\mathsf X$ is a.s. periodic for the same reasons as $C$. Finally, $\mathsf  Z _{\mathsf X}$ is completely predictable because $C(x)$ gets arbitrarily close to $1$ as $x\to \infty $, hence the previous proposition  applies.
\end{proof}

Even if such processes could be useful for modelling pseudo-crystalline arrangments, these periodicity properties prevent them from having the slightest form of asymptotic independence, ergodicity, let alone mixing properties.
We shall give in the next section examples of fields  whose zero set is weakly mixing and completely predictible. 

The excursion volume of almost periodic fields have interesting properties. It is proved in the forthcoming paper \cite{Lr20+} that  $\mathcal  V_{\mathsf X}(T)=\Var(|[-T,T]\cap \{\mathsf X>0\}|)$
 strongly depends on the diophantine properties of the $\varphi _{i}$, where $ | \cdot  | $ denotes Lebesgue measure. If for instance $C_{\mathsf X}(x)=\cos(x)+\cos(\varphi  x)$, $\mathcal  V_{\mathsf X}(T)$ can either be bounded, or grow as an arbitrarily high power of $T$, depending on weather $\varphi $ has ``good'' or ``bad'' approximations by rational numbers.

\section{Symmetric Bernoulli convolutions}
\label{sec:bernoulli}

Let $\lambda =(\lambda _{k})$ be a square summable family of $\mathbb{R}_{+}$, and 
\begin{align*}
C^{\lambda }(t)=\prod_{k=1}^{\infty }\cos(\lambda _{k}t ), t\in \mathbb{R}.
\end{align*}
This function converges for every $t$   (\cite[(3.7.9)]{Luk70}), and it is the reduced covariance of some stationary Gaussian  field denoted by $\mathsf X^{\lambda }$. The order of multiplication does not matter in the value $C^{\lambda }(t)$, hence we can assume throughout  that $\lambda _{k}\geqslant \lambda _{k+1}$. By studying the uniform convergence on every compact, it is  clear that $\mathsf X^{\lambda }$ is analytic, with an analytic continuation to $\mathbb  C$.

The spectral measure $\mu ^{\lambda }$ of $\mathsf X^{\lambda }$ is called a {\it symmetric Bernoulli convolution} because it is the  law of the random variable
\begin{align*}
Y^{\lambda }:=\sum_{k=1}^{\infty }\lambda _{k}{\varepsilon _{k}}
\end{align*}
where the $\varepsilon _{k}$ are iid Rademacher variables. Infinite products of characteristic functions have been studied by Lukacs \cite[Chapter 3.7]{Luk70}, where basic results stated here are proved. The main focus of study in the literature is the nature of   $\mu ^{\lambda }$. It is known that $\mu ^{\lambda }$ is {\it pure}, i.e. either continuous, purely discrete, or purely singular \cite[Th. 3.7.7]{Luk70}. Bernoulli convolutions with geometric progression are a very active research subject, see for instance \cite{Varju}, they are one of the most studied examples of self-similar measures, and they are objects of great interest in fractal geometry.

Regarding ergodicity, let $\{\mathsf X(t);t\in \mathbb{R}\}$ be a Gaussian stationary process with almost surely continuous paths, and $\mathscr {B}=\sigma (\mathsf X)$ the $\sigma $-algebra generated by the random variables $\mathsf X(t),t\in \mathbb{R}$. For $B\in \mathscr {B},t\in \mathbb{R}$, let $B+t$ be the event $B$ translated by $t$. The process $\mathsf X$ is called {\it weakly mixing} if 
\begin{align*}
\lim_{T\to \infty }T^{-1}\int_{0}^{T}\left|
\P (A\cap (B+t))-\P (A)\P (B)
\right|dt=0.
\end{align*}

\begin{proposition}
Assume
 $\lambda _{n}>R _{n}>0$ for every $n.$ Then  $\mu ^{\lambda }$ has no atoms, its support has Cantor type (compact, completely disconnected and without isolated points), and
\begin{align*}
\lim_{T\to \infty }\frac{1}{T}\int_{0}^{T}C^{\lambda }(t)dt=0,
\end{align*}
hence $\mathsf X^{\lambda }$ and its zero set $\mathsf  Z ^{\lambda }$ are weakly mixing.
\end{proposition}

\begin{proof}The support of $\mu $ is 
\begin{align*}
\Sigma =\{\sum_{k}\lambda _{k}\varepsilon _{k}: \varepsilon =(\varepsilon _{k})\in \{-1,1\}^{\mathbb{N}}\}.
\end{align*}
Hence it is clear that given $t=\sum_{k}\lambda _{k}\varepsilon _{k}\in \Sigma ,\eta  =\lambda _{n}$ for some $n\in \mathbb{N}$, 
\begin{align*}
(\Sigma \setminus \{t\})\cap [t-\eta  ,t+\eta  ]\supset \{\sum_{k=1}^{n}\lambda _{k}\varepsilon _{k}+\sum_{k>n}\lambda _{k}\varepsilon '_{k-n}:\varepsilon '\in \{-1,1\}^{\mathbb{N}}\}
\end{align*}is infinite. It follows that $t$ is an accumulation point of $\Sigma $. To prove that $\Sigma $ is completely disconnected, let $t=\sum_{k}\varepsilon _{k}\lambda _{k},s=\sum_{k}\varepsilon _{k}'\lambda _{k}\in \Sigma $ with $t<s$, and $k_{0}=\min\{k:\varepsilon _{k}\neq \varepsilon _{k}'\}$. Then let
\begin{align*}
t_{+}=\sum_{k<k_{0}}\varepsilon _{k}\lambda _{k}+\sum_{k\geqslant k_{0}}\lambda _{k},\;s_{-}=\sum_{k<k_{0}}\lambda _{k}\varepsilon _{k}-\sum_{k\geqslant k_{0}}\lambda _{k}.
\end{align*}
We have $t\leqslant t_{+}<s_{-}\leqslant s$, where the strict inequality comes from $R_{k_{0}}<\lambda _{k_{0}}$. Noticing that there is no point of $\Sigma $ in $(t_{+},s_{-})$, $t$ and $s$ are not connected through $\Sigma $.

To prove that there is no atom, symmetry considerations yield that $\mu (\{t\})=\mu (\{s\})$ for every $t,s\in \Sigma $; since $\mu $ is finite, it follows that $\mu (t)=0$ for every $t\in \Sigma .$ Hence according to Maruyama's theorem \cite{Maruyama}, $\mathsf X$ is ergodic, and according to  Zak \& Rosi\'nski \cite{RZ97}, it is hence weakly mixing.

 Since $\mathsf X$ is a.s. continuous, for $I\in \mathscr{B}(\mathbb{R})$ open, we have a.s.
\begin{align*}
\#\mathsf  Z _{\mathsf X}\cap I=\lim _{n\geqslant 1}\sum_{k=-\infty }^{\infty }\mathbf{1}_{\{\mathsf X(k/n)\mathsf X((k+1)/n)\leqslant 0\}},
\end{align*}
hence any event $A\in \sigma (\mathsf  Z _{\mathsf X})$ is $\sigma (\mathsf X)$-measurable. It readily follows that since $\mathsf X$ is weakly mixing, so is $\mathsf  Z _{\mathsf X}.$ 
\end{proof}

There are also many cases where $C^{\lambda }$ has stronger mixing properties.
Most available results concern  coefficients of the form $\lambda _{k}=a^{k} $ for some $a\in (0,1)$. The case $a=1/2,$ i.e. $\lambda _{k}=2^{-k}$, actually yields  the  {\it sine process} thanks to Vieta's formula:
\begin{align*}
 \frac{\sin(t)}{t}=&\frac{2\sin(t/2)\cos(t/2)}{t}=\frac{4\sin(t/4)\cos(t/4)\cos(t/2)}{t}\\
=&\dots \\
=&\lim_{n}\frac{2^{n}\sin(2^{-n}t)}{t}\prod_{k}\cos(2^{-k}t)=C^{\lambda }(t).
\end{align*}
This example is central, as it is the only $a$ for which $R_{n}=\lambda _{n}$, and it is the smallest $a$ for which $\mu ^{\lambda }(dx)=\mathbf{1}_{\{[-1,1]\}}(x)dx$ is continuous. Since $C^{\lambda }(t)\to 0$ as $t\to \infty $, $\mathsf X^{\lambda }$ and its zero set $\mathsf  Z ^{\lambda }$ are (strongly) mixing.  For $a<1/2$, the previous considerations yield that $\mu ^{\lambda }$ is Cantor-like with no atoms, hence weakly mixing. If $a$ is not of the form $m^{-k}$ for some $m\in \mathbb{N}$, $C^{\lambda }(t)=O( | \log(t) | ^{-\gamma })$ for some $\gamma >0$ as $t\to \infty $, whereas for $a=m^{-k},m\geqslant 3$, $\limsup_{t\to \infty }C^{\lambda }(t)>0$, see \cite{Luk70}.

There are many other examples of such products, with diverse  behaviours.
If for instance  $\lambda _{k}=k^{-1},$
\begin{align*}
C^{\lambda }(t)=\prod_{k}\cos\left(
\frac{t}{k}
\right)=\prod_{n} \sinc\left(
\frac{t}{2n+1}
\right)
\end{align*}
decreases faster than any polynomial.
Let us give now an example of symmetric Bernoulli convolution which has rigid zeros. As a further proof of asymptotic independence, we show that this process is a.s. unbounded, at the contrary of almost periodic Gaussian fields mentioned in the previous section.

\begin{proposition}
\label{prop:factorial}
Assume $\lambda _{k}=k !^{-1}\pi $. Then $(-1)^{n}C^{\lambda }(n!)\to 1$ as $n\to \infty $. In particular, the zero set $\mathsf  Z $ of the Gaussian process $\mathsf X$ with reduced covariance $C^{\lambda }$ is completely predictable, hence maximally rigid. Furthermore, $\mu _{\lambda }$ is purely singular without atoms, hence $\mathsf X$ and $\mathsf  Z $ are weakly mixing. Also, $\sup_{t\in \mathbb{R}} | \mathsf X_{\lambda }(t) | =\infty $ a.s.
\end{proposition}

\begin{proof}
Let $t_{n}=n!$.  We have
\begin{align*}
\prod_{k=1}^{n}\cos(t_{n}\lambda _{k})=\prod_{k=1}^{n}\cos(\pi n!/k!)=(-1)^{n}.
\end{align*}For  $k\geqslant n+1$,  
\begin{align*}
\cos\left(
\frac{\pi n!}{k!}
\right)\geqslant 1-\left(
\frac{\pi n!}{k!}
\right)^{2}\geqslant 1-\left(
\frac{\pi ^{2}}{k^{2}}
\right)
\end{align*}
hence  the first result is proved with
\begin{align*}
1\geqslant (-1)^{n}C(t_{n})\geqslant \prod_{k=n+1}^{\infty }\left(
1-\left(
\frac{\pi }{  k}
\right)^{2}
\right)\geqslant 1-c/n.
\end{align*}
 
 To prove the unboundedness of $\mathsf X^{\lambda }$, 
 consider the pseudo-metric on $\mathbb{R}$
\begin{align*}
d(t,s)=\sqrt{1-C(t-s)}.
\end{align*}
We will prove that $(\mathbb{R},d)$ cannot be covered by finitely many $d$-balls of some radius $\varepsilon >0$. Let $t\in B_{d}(0,\varepsilon )$, and for $k\in \mathbb{Z} $, let $n_{k}=n_{k}(t)$ be the integer such that for some $\varepsilon _{k}=\varepsilon _{k}(t)\in [-1/2,1/2)$
\begin{align*}
 {k!^{-1}t}=n_{k} +\varepsilon _{k}.
\end{align*}Then  for some $c>0$
\begin{align*}
d(t,0)^{2}\geqslant 1-\exp(\sum_{k} \ln( | \cos(\varepsilon _{k}\pi )) | )\geqslant c\sum_{k}\varepsilon _{k}^{2}
\end{align*}
hence $c\sum_{k}\varepsilon _{k}^{2}<\varepsilon ^{2}$, and in particular each $\varepsilon _{k}$ is smaller than $\delta :=\sqrt{c^{-1}}\varepsilon $. Up to diminishing $\varepsilon $, assume $\delta <1/4.$ Let us now cover $\mathbb{R}_{+}$ by the sets 
\begin{align*}
A_{m,k}=\{t=mk!+k!\varepsilon _{k}:\varepsilon _{k}\in [-1/2 ,1/2 ]\}
\end{align*}
for $k\in \mathbb{N} ,0<m\leqslant  k .$ We have for $t\in A_{m,k},0\leqslant l\leqslant k$, with $k^{(l)}=k(k-1)\dots (k-l+1),$
\begin{align*}
\frac{t}{(k-l)!}=mk^{(l)}+k^{(l)}\varepsilon _{k}=n_{k-l}(t)+\varepsilon _{k-l}(t),
\end{align*}
hence $\varepsilon _{k}\in C_{k^{(l)} }$, where 
\begin{align*}
C_{K }=\{t \in [-1/2,1/2]:d(Kt ,\mathbb{Z} )\leqslant \delta  \}, K>0.
\end{align*}
Let us then define $E_{0}=[-\delta ,\delta ]$ and for $0\leqslant l<k$,
\begin{align*}
E_{l+1}=E_{l}\cap C_{k^{(l+1)}}.
\end{align*}
Let $n_{l}$ be the number of maximal segments forming $E_{l}$. We prove by induction that $n_{l}\leqslant (2\delta )(4\delta )^{l-1}k^{(l)}$ for $k-l\geqslant \delta ^{-1}$, i.e. $l\leqslant l^*:=[k-\delta ^{-1}]$. Indeed, through the intersection with $C_{k^{(l)} }$, $E_{l}$ is made up of segments of length at most $a_{l}:=2\delta k^{-(l)}$, and each segment of length $a_{l}$ is intersected by at most $$a_{l}k^{(l+1)}+2=2\delta (k-l)+2\leqslant 4\delta (k-l)$$ segments of $C_{k^{(l+1)}}.$ It follows that 
\begin{align*}
n_{l+1}\leqslant n_{l}4\delta (k-l)\leqslant 2\delta (4\delta )^{l}k^{(l)}(k-l);
\end{align*}
which proves the induction. Hence the Lebesgue measure of $E_{l}$ for $l\leqslant l^*$ is smaller than $2\delta (4\delta )^{l}k^{(l)}2\delta k^{-(l+1)}\leqslant (4\delta )^{l+1}.$
We have for $l\geqslant l^*$, $E_{l}\subset  E_{l^{*}}$. 
It follows that
\begin{align*}
 | B_{d}(0,\varepsilon )\cap A_{m,k} | \leqslant &  | E_{k} | \leqslant  |  E_{l^{*}} | \leqslant (4\delta )^{l^{*}+1}.\\
\end{align*}
It  follows that for $\delta <1/4$
\begin{align*}
 | B_{d}(0,\sqrt{\varepsilon }) | \leqslant \sum_{k}\sum_{m}(4\delta )^{k-\delta ^{-1}-1} <\infty .
\end{align*}
By stationarity, $ | B_{d}(t,\varepsilon ) | = | B_{d}(0,\varepsilon ) | <\infty $ for $t\in \mathbb{R}$, hence $\mathbb{R}$ cannot be covered by a finite number of such balls, hence the sample paths of $\mathsf X^{\lambda }$ are a.s. unbounded (see for instance \cite[Th. 1.19]{AW}).

\end{proof}
%
%
%
%
%
%
%
%
%

\section{Variance lower bounds}
\label{sec:chaotic}

In this section we take $\mathsf X$ as a stationary Gaussian process whose reduced covariance function $C$ satisfies Geman's condition \eqref{eq:geman}. In particular,  $C'$ and $C''(0)$ exist, hence Fatou's lemma yields 
 that the spectral measure $\mu $ has a finite second moment, and 
\begin{align}
\label{eq:second-deriv}
C''(t)=-\int_{}t^{2}e^{ itx}\mu (dx), t\in \mathbb{R}.
\end{align} Recall that $C$ is normalised so that $C''(0)=-1.$ 
If Geman's condition is not satisfied (and $\varphi \geqslant 0$), the variance is infinite because the support of $\varphi $ has non-empty interior, and the result of the theorem holds true.

Let us start by the case where $C(\tau )=\pm 1$ for some $\tau >0$. It implies that $\mu $ is supported by $\frac{2\pi }{\tau }\mathbb{Z} $ and has (at least) an atom, and  that $\mathsf  X$ is a.s. $2\tau $-periodical. The behaviour of $N_{X}(\varphi _{T})$ as $T\to \infty $ is determined by the zeros $z_{1},\dots ,z_{Q}\in \mathsf  Z_{\mathsf  X}\cap [0,2\tau )$, for some random $Q\in \mathbb  N.$ We have 
\begin{align*}
T^{-1}N_{\mathsf  X}(\varphi _{T})=\sum_{i=1}^{Q}T^{-1}\sum_{k\in \mathbb{Z} }\varphi (T^{-1}(z_{i}+2k\tau ))=\sum_{i=1}^{Q}(2\tau )^{-1}\int_{}\varphi (t)dt+o(1)
\end{align*}
due to the hypotheses on $\varphi .$ If we upper bound $\varphi $ by $\|\varphi \|$ on its support, it easily follows by Lebesgue's Theorem that $T^{-2}V_{\mathsf  X}(\varphi _{T})$ converges to $\Var(Q)(2\tau )^{-2}\left(
\int_{}\varphi 
\right)^{2}$. The only case where $\Var(Q)=0$ is when $C(t)=\cos(t) $ up to rescaling and multiplication, which is the degenerate case excluded from the theorem. Hence in all other cases we have indeed a quadratic variance, and point (iii) is proved in this case.

Let us now assume that $C(t)\neq \pm 1$ for $t\neq 0.$ Following Kratz \& L\'eon, it is proved in \cite{AW} that we have the decomposition  
$$N_\mathsf X({T})=\E (N_\mathsf X({T}))+\sum_{q=1}^{\infty }N_{\mathsf  X,q}(T)$$ with
\begin{align*}
N_{\mathsf  X,q}(T)=\sum_{k=0}^{q}a _{k}d _{q-k}\int_{0}^{T}H_{k}(\mathsf X(t))H_{q-k}(\mathsf X'(t))dt
\end{align*}
where $H_{m},m\in \mathbb{N}$ is the $m$-th Hermite polynomial, and $a _{m},d_{m},m\in \mathbb{N}$ are coefficients that vanish for $m$ odd. In particular,
\begin{align*}
a_{0}=\frac{1}{\sqrt{2\pi }}, &\;a_{2}=\frac{1}{2\sqrt{2\pi }}\\
d_{0}=1,&\;d_{2}=-\frac{1}{2}
\end{align*}and
 $N_{\mathsf  X,q}(T)$ vanishes for odd $q.$ (The decomposition also holds if we only assume that $C$ is four times differentiable in $0$, see \cite{Kratz-survey}).

By linearity, the decomposition still holds if $\varphi $ is a piecewise constant compactly supported function, and $N_{\mathsf  X }(\varphi )$ has the $L^{2}$-orthogonal decomposition $\sum_{q=0}^{\infty }N_{\mathsf  X,q}(\varphi )$, with
\begin{align*}
N_{\mathsf  X,2 }(\varphi )=\frac{-1}{2\sqrt{2\pi }}\left[
\int_{\mathbb{R}}\varphi (t)\left(
H_{0}(\mathsf X(t))H_{2}(\mathsf X'(t))dt-\int_{-T}^{T}H_{2}(\mathsf X(t))H_{0}(\mathsf X'(t))dt
\right)
\right].
\end{align*}

Then, with $H_{0}=1,H_{2}(x)=x^{2}-1$, since for two standard Gaussian variables $\alpha ,\beta $ with correlation $\rho $, $\Cov(\alpha ^{2},\beta ^{2})=2\rho ^{2}$,
\begin{align*}
V_{\mathsf  X}(\varphi )
\geqslant &\Var (N_{\mathsf X  ,2}(\varphi ))\\
=&\frac{1}{4\pi }\int_{\mathbb{R}^{2}}\varphi (t)\varphi (s)\left[\Cov (\mathsf X(t)^{2},\mathsf X(s)^{2})+\Cov (\mathsf X'(t)^{2},\mathsf X'(s)^{2})-
2\Cov (\mathsf X(t)^{2},\mathsf X'(s)^{2}) 
\right]dtds.
\end{align*}
We can show that this bound can be extended to any $\varphi \in \mathcal{C}_{b}$ by using Fatou's lemma and an approximating sequence of piecewise constant functions $\varphi _{n},n\geqslant 1$.

Then, with $z=t-s,w=-s,$
\begin{align*}
\Var (N_\mathsf X(T))\geqslant&\frac{1}{4\pi } \int_{\mathbb{R}^{2}}[C(z)^{2}+C''(z)^{2}-2C'(z)^{2}]\left[
\int_{\mathbb{R}}\varphi ( {z-\omega}  )\varphi ( { \omega}  )dw
\right]dz\\
=&\frac{1}{4\pi }\int_{\mathbb{R}^{2}}[C(z)^{2}+C''(z)^{2}-2C'(z)^{2}]\varphi ^{\star 2}(z)dz
\end{align*}
where $\star $ denotes the convolution product. Denote by 
\begin{align*}
\hat \psi  (x)=\int_{\mathbb{R}}e^{ixt}\psi  (t)dt,x\in \mathbb{R},
\end{align*}the Fourier transform of a $L^{2}$ function $\psi $. By \eqref{eq:second-deriv},
\begin{align*}
\int_{\mathbb{R}}C(z)^{2}\varphi ^{\star 2}(z)dz=&\int_{\mathbb{R}}\varphi ^{\star 2} (z)\int_{\mathbb{R}^{2}}e^{ixz}e^{iyz}\mu (dx)\mu (dy)dz\\
=&\int_{\mathbb{R}}\hat\varphi (x+y)^{2}\mu (dx)\mu (dy)\\
\int_{\mathbb{R}}C'(z)^{2}\varphi ^{\star 2}(z)dz=&\int_{\mathbb{R}}\varphi ^{\star 2} (z)\int_{\mathbb{R}^{2}}(ix)(iy)e^{ixz}e^{iyz}\mu (dx)\mu (dy)dz\\
=&-\int_{\mathbb{R}}xy\;\hat\varphi (x+y)^{2}\mu (dx)\mu (dy)\\
\int_{\mathbb{R}}C''(z)^{2}\varphi ^{\star 2}(z)dz=&\int_{\mathbb{R}}\varphi ^{\star 2} (z)\int_{\mathbb{R}^{2}}x^{2}y^{2}e^{ixz}e^{iyz}\mu (dx)\mu (dy)dz\\
=&\int_{\mathbb{R}}x^{2}y^{2}\,\hat\varphi (x+y)^{2}\mu (dx)\mu (dy)\\
\end{align*}
whence finally,  
\begin{align*}
V_{\mathsf  X}(\varphi )\geqslant & \frac{1}{4\pi }\int_{\mathbb{R}^{2}}(1+xy)^{2}\hat\varphi (x+y)^{2}\mu (dx)\mu (dy).
\end{align*}
Then remark that $\hat\varphi _{T}(x)=T\hat \varphi (Tx)$. Since $\int_{}\varphi \neq 0$, $\hat \varphi (0)\neq 0$, by continuity  there is $\alpha >0$ such that $ | \hat \varphi (x) | \geqslant  | \int{\varphi } | /2>0$ for $ | x | \leqslant \alpha .$ It simplifies the argument (without loss of generality) to rescale $\varphi $ so that $\alpha =1$. Then, for some $c_{\varphi }>0,$
\begin{align}
\label{eq:var-phi-mu}
\Var(N_{\mathsf  X}(\varphi _{T}))\geqslant &c_{\varphi }T^{2}\int_{\mathbb{R}}\mathbf{1}_{\{ | T(x+y) | <1 \}}(1+xy)^{2}\mu (dx)\mu (dy).
\end{align}
If $\mu $ has an atom $x_{0}\notin \{-1,1\}$, the right hand side is larger than $c_{\varphi }T^{2}(1-x_{0}^{2})^{2}\mu (\{x_{0}\})^{2}$, hence Theorem \ref{thm:main-variance}-(iii), is proved.

Let now $\varepsilon >0$ and  a set $A$ being either $$(-\infty ,-1-\varepsilon ],\;[-1+\varepsilon ,1-\varepsilon ] \text{\rm{ or }}[1+\varepsilon ,\infty ),$$ the choice being made so that, with $ \mu _{\varepsilon }=\mu \mathbf{1}_{\{A\}}$:   $\mu_{\varepsilon }\neq 0 $ (recall that $\mu $ was assumed to not be concentrated on $\{\pm1\}$), and if $\mu $ does not have a density on $\mathbb{R}\setminus\{-1,1\}$, $\mu_{\varepsilon } $ does not have a density  either. Also define
\begin{align*}
C_{\varepsilon }(t)=\int_{\mathbb{R}}e^{ixt}\mu _{\varepsilon }(dx).
\end{align*}For $T>2 \varepsilon ^{-1}$, we have for some $c_{\varepsilon }>0$,  for $x,y\in A,$
\begin{align*}
(1+xy)^{2}\mathbf{1}_{\{ | T(x+y) | <1 \}}
\geqslant& c_{\varepsilon }\mathbf{1}_{\{ | T(x+y) | <1 \}}\geqslant c_{\varepsilon }\Delta _{T}(x+y)
\end{align*} 
where $ \Delta _{T}(x):=\mathbf{1}_{\{ |T x | <1\}}(1- | Tx | ).$
The Fourier inversion formula yields for $t\in \mathbb{R}$
 $${\Delta _{T}}(x)=\frac{2T^{-1}}{\sqrt{2\pi }}\int_{}\sinc(T^{-1}t)^{2}e^{2itx}dt,$$ 
hence \eqref{eq:var-phi-mu} yields for some $c_{\varphi }'>0$
\begin{align}
\notag\Var(N_\mathsf X(\varphi _{T}))\geqslant &c_{\varphi }T^{2}\int_{}\int_{}\Delta _{T}(x+y)\mu_{\varepsilon }(dx)\mu_{\varepsilon } (dy) \\
\notag=&c_{\varphi }'T\int_{}\int_{}\int_{}e^{2ixt}e^{2iyt}\sinc(T^{-1}t)^{2}\mu _{\varepsilon }(dx)\mu _{\varepsilon }(dy)\\
\notag=&c_{\varphi }'T \int_{\mathbb{R}}\sinc^{2}(T^{-1}t)C_{\varepsilon }(2t)^{2}dt\\
\label{eq:lin}\geqslant &c_{\varphi }'\sin(1)^{2}T \int_{\mathbb{R}}\mathbf{1}_{\{ | t | <T\}}C_{\varepsilon }(2t)^{2}dt
\end{align}
and we have used the finiteness of $\mu _{\varepsilon }$ and the integrability of $\sinc^{2}.$ Since by construction $\mu _{\varepsilon }$ has positive mass, $C_{\varepsilon }$ is not identically zero and Theorem \ref{thm:main-variance}-(i) is proved.  

In view of proving (ii), we assume that the variance is linear, hence we need to show that  $C''+C\in L^{2}(\mathbb{R})$.  By \eqref{eq:lin}, $C_{\varepsilon }$ is $L^{2}$, hence $\mu _{\varepsilon }$ has a $L^{2}$ density, and $\mu $ also by construction of $\mu _{\varepsilon }$. Let $f$ be the density of $\mu $.
 By \eqref{eq:var-phi-mu}, (after changing $y$ to $-y$ and using the symmetry of $\mu $)
\begin{align*}
\Var(N_\mathsf X(\varphi _{T}))\geqslant &c_{\varphi }T^{2}\int_{}\int_{}(1-xy)^{2}\mathbf{1}_{\{ | x-y | <T^{-1}\}}f(x)f(y)dxdy\\
\liminf_{T}T^{-1}\Var(N_\mathsf X(\varphi _{T}))\geqslant &c_{\varphi }\int_{\mathbb{R}\setminus\{-1,1\}}\left(
 \liminf_{T}T\int_{x-1/T}^{x+1/T}(1-xy)^{2}f(y)dy
\right)f(x)dx\\
=&c_{\varphi }\int_{}(1-x^{2})^{2}f(x)^{2}dx\\
=& c_{\varphi }\int_{}(C''+C)^{2}
\end{align*}
by Parseval's identity, hence $C''+C\in L^{2}$.

To conclude the proof of (ii), let us assume now that $\mu $ has a $L^{2}$ density $f$ on a bounded neighbourhood $I$ of $\{-1,1\}$, and show that \eqref{eq:C2-L2} is equivalent to $C+C''\in L^{2}$. If $\mu $ does not have a density on all $\mathbb{R}$, then $C\notin L^{2}$, otherwise $\widehat C$ would be a $L^{2}$ density of $\mu $. Then, if $C+C''\in L^{2}$, Fourier's inversion formula and \eqref{eq:second-deriv} yield 
\begin{align*}
C(t)+C''(t)=\frac{1}{\sqrt{2\pi }}\int_{}\widehat {(C+C'')}(x)e^{itx}dx= \int_{}(1-x^{2})e^{itx}\mu (dx), t\in \mathbb{R},
\end{align*}
hence $\mu $ has density $\widehat {C+C''}(x)(1-x^{2})^{-1}$, which leads to a contradiction. Hence none of the conditions hold if $\mu $ does not have a density.
 
Let us assume finally  that $\mu $'s density $f $ can be extended to all $\mathbb{R}$. Then $$C(t)=\int_{}e^{ixt}f(x)dx,\;\;C''(t)=-\int_{}x^{2}f(x)e^{ixt}dx,\;\;(C+C'')(t)=\int_{}(1-x^{2})f(x)e^{ixt}dx.$$
The integrands are square integrable on $I$, so we need to prove the equivalence on $B:=\mathbb{R}\setminus I$. We have $$C\in L^{2}(B)\Leftrightarrow f \in L^{2}(B),C''\in L^{2}(B)\Leftrightarrow \int_{B}x^{4}f(x)^{2}dx<\infty ,\;\;C''+C\in L^{2}(B)\Leftrightarrow \int_{B}(1-x^{2})^{2}f^{2}(x)dx<\infty,$$ hence the equivalence stems from the existence of $c>0$ such that for $x\notin I$
 \begin{align*}
c(1+x^{4})\leqslant (1-x^{2})^{2}\leqslant 1+x^{4}.
\end{align*}

 \begin{remark}
 One might think that going for higher order chaoses could help replace (ii) by the sufficient condition $C,C''\in L^{2}$, but the gap present at $\pm 1$ is necessarily still present at higher order chaoses because the zeros of $C(t)=\cos(t)$ satisfy the same chaotic decomposition (see \cite[Prop. 2.2]{Kratz-survey}), hence the variance of every chaos is bounded in $T.$ It remains open to find a necessary and sufficient condition for linear variance.
\end{remark}

\begin{proposition}
\label{prop:variance-factorial}
If
\begin{align*}
C_{\mathsf  X}(t)=\prod_{k}\cos(t/k!),
\end{align*}
then  for all $\varepsilon >0,$
\begin{align*}
T^{\varepsilon -2}\Var (N_\mathsf X(T))\to \infty 
\end{align*}
as $T\to \infty .$
\end{proposition}

\begin{proof}
Let 
\begin{align*}
R_{N}=\sum_{k>N}k!^{-1}.
\end{align*}
Recall that the spectral measure $\mu $ is the law of $Y=\lim_{N}Y_{N}$ where $Y_{N}=\sum_{k=1}^{N}\frac{\varepsilon _{k}}{k!}$, for random iid Rademacher variables $(\varepsilon _{k})$.
In particular,  $|Y-Y_{N}|<R_{N}$ a.s. Two consecutive points of the support of $Y_{N}$ have distance $2/n!>2R_{N}$ (proved by induction) and they all have mass $2^{-N}$.  Let $Y',Y'_{N},N\in \mathbb{N}$ be iid copies of $Y,Y_{N}$.
Since $Y$'s support is in $(-1,1)$, there is $c>0$ such that for $T$ sufficiently large, using \eqref{eq:var-phi-mu},
\begin{align*}
\Var(N_\mathsf X(T))\geqslant c T^{2}\P (  | Y-Y' | <1/T ).
\end{align*}
Let $N=N_T$ be the  integer such that $R_{N-1}\geqslant \frac{1}{4T}>R_{N}$. 
\begin{align*}
\P (|Y-Y'|<1/T)\geqslant &\P (|Y_{N}-Y'_{N}|<2R_{N})\\
= & \sum_{x\in \text{\rm{Supp}}(Y_{N})}\P (Y_{N}=Y'_{N}=x)\\
= &2^{-N}.
\end{align*}
Then we have by convexity
\begin{align*}
\ln(4T)\geqslant \ln(R_{N-1}^{-1})>\ln((N-1)!)=\sum_{k=1}^{N-1}\ln(k)>(N-1)\ln((N-1)/2).
\end{align*}
It follows that 
\begin{align*}
(4T)^{\varepsilon -2}\Var(N_{\mathsf X}(T))\geqslant \exp[\varepsilon(N-1)\ln((N-1)/2))]2^{-N}\to \infty .
\end{align*}

\end{proof}

\section*{Acknowledgements} I warmfully thank Anne Estrade, with whom I had many discussions along the elaboration of this work, about Gaussian processes and their zeros. I also wish to thank Jose L\'eon for a discussion around the variance of Gaussian zeros.

This work was partially funded by the grant {\it Emergence en recherche} from the Idex {\it Universit\'e de Paris}.

\bibliographystyle{plain}
\bibliography{../../bibi2bis}

\end{document}